\documentclass[12pt,a4paper]{amsart}

\usepackage{amsmath}
\usepackage{amsthm}
\usepackage{amssymb}

\newtheorem{definition}{Definition}[section]

\newtheorem{corollary}{Corollary}[section]
\newtheorem{theorem}{Theorem}[section]

\newtheorem{example}{Example}[section]

\begin{document}
\title[On the stability of the Bishop's property ($\beta$)]
{On the stability of the Bishop's property($\beta$) under compact
perturbations}
\author[S.Mecheri ]{Salah Mecheri}
\address{Department of Mathematics\\
Faculty of Science and Informatics\\
El Bachir Ibrahimi University, Bordj Bou Arreridj, Algeria.}
 \email{salahmecheri20@gmail.com}
\dedicatory{} 
\subjclass[2000]{Primary 47B47, 47A30, 47B20;\\
\hspace{4mm} Secondary 47A15.}
\begin{abstract}
Let $B(X)$ be the Banach algebra of all bounded linear operators
acting on a Banach space $X$. Are sums and products of commuting
decomposable operators on Banach spaces decomposable? This is one
of the most important open problems in the local spectral theory
of operators on Banach spaces. Similarly, it is not known if local
spectral properties such as the single valued extension property,
Dunfords property $(C)$, Bishops property $(\beta)$, or the
decomposition property ($\delta$) are preserved under sums and
products of commuting operators. But it is shown by Bourhim and
Muller that the single-valued extension property is not preserved
under the sums and products of commuting operators. On the
positive side, Sun proved that the sum and the product of two
commuting operators with Dunfords property $(C)$ have the
single-valued extension property. Very recently, Aiena and Muller
showed that the (localized) single-valued extension property is
stable under com- muting Riesz perturbations. In this paper, we
show that Bishops property ($\beta$), the decomposition property
($\delta$), or decomposable operators $T\in B(X)$ are stable under
quasinilpotent, compact, and algebraic commuting perturbations.
\end{abstract}
 \maketitle
\section{Introduction}
Let $B(X)$ be the Banach algebra of all bounded linear operators
acting on a Banach space $X$. An operator $T \in B(X)$ is said to
have the single valued extension property at $\lambda_{0}\in
\mathbb{C}$ (abbreviated SVEP at $\lambda_{0})$, if for every open
disc $D$ of $\lambda_{0}$, the only analytic function $f : D
\Rightarrow X$ which satisfies the equation $(\lambda
I-T)f(\lambda) = 0$ for all $\lambda\in D$ is the function
$f\equiv 0$. An operator $T \in B(X)$ is said to have SVEP if T
has SVEP at every point $\lambda\in\mathbb{C}$. The study of
operators satisfying Bishops property $(\beta)$ is of significant
interest and is currently being done by a number of mathematicians
around the world \cite{17,18,19,20,21}. Let $D(\lambda; v)$ be the
open disc centred at $\lambda\in\mathbb{C}$ with radius $v
> 0$, and let $O(U;X)$ denote the Fr\'echet algebra of all
$X$-valued functions on the open subset of $U\in\mathbb{C}$
endowed with the uniform convergence on compact subsets of $U$. An
operator $T\in B(X)$ has Bishop's property $(\beta)$ at
$\lambda_{0}$ if there exists $v>0$ such that for every open
subset $U \subset D(\lambda_{0}; v)$ and for any sequence
$\{f_{n}\}^{\infty}_{n=1} \subset O(U;X)$, $\lim_{n\rightarrow
\infty}(S-\lambda I)f_{n}(\lambda) = 0$ in $O(U;X)$ implies
$\lim_{n\rightarrow \infty} f_{n}(\lambda) = 0$ in $O(U;X)$. An
operator $T\in B(X)$ is said to have the decomposition property
$(\delta)$ if $T^{ *}$ satisfies Bishops property $(\beta)$, where
$T^{*}$ is the dual operator of $T$. In \cite{6}, Foias showed
that every decomposable operator (and therefore spectral operators
in the sense of Dunford, all generalized scalar operators in the
sense of Colojoara and Foias \cite{6}, compact operators, and
unitary, normal, and self-adjoint operators on a Hilbert space)
has decomposition property ($\delta$). It is well known that $T$
is decomposable if and only if $T$ satisfies both Bishop's
property $(\beta)$ and decomposition property $\delta)$. The left
shift operator $L$ on $l_{2}(N)$ has the decomposition property
$(\delta)$. Indeed, since the right shift operator $R$ on
$l_{2}(N)$ is subnormal as the restriction of the bilateral right
$R$ shift on $l_{2}(Z)$. Clearly, $R$ has the Bishop's property
$(\beta)$. Since $L$ is the adjoint of $R$, it follows that $L$
has the decomposition property ($delta$). The left shift operator
$L$ on $l_{2}(N)$ is an example of a bounded linear operator that
has decomposition property $\delta)$, but whose adjoint does not.
This shows that the decomposition property $\delta)$ is not
preserved under the adjoint operation. The natural related
operator in the context of the spectral theory is the restriction
operator. We give an example of an operator $T$ and a
$T$-invariant subspace $M$ such that $T$ has decomposition
property ($\delta$); but $T|_{M}$ does not. Let $T$ be the right
bilateral shift operator on $l_{2}(\mathbb{Z})$; and let $Y :=
span\{e_{i} : i = 1; 2\}^{-}$, where $\{e_{i} : i\in \mathbb{Z}\}$
is the usual orthonormal basis for $l_{2}(\mathbb{Z})$. Now T is
unitary and so certainly has decomposition property $(\delta)$,
but $T|_{M}$ is isomorphic to the right shift on $l_{2}(N)$, and
hence does not have decomposition property $(\delta)$. In
particular, the single valued extension property of operators was
first introduced by N. Dunford to investigate the class of
spectral operators which is another important generalization of
normal operators (see \cite{8}). In the local spectral theory, for
given an operator T on a complex Banach space $X$ and a vector $x
\in X$, one is often interested in the existence and the
uniqueness of analytic solution $f(.) : U \rightarrow X$ of the
local resolvent equation $$(T -\lambda)f(\lambda) = x$$ on a
suitable open subset $U$ of $\mathbb{C}$. Obviously, if $T$ has
SVEP, then the existence of analytic solution to any local
resolvent equation (related to $T$) implies the uniqueness of its
analytic solution. The SVEP is possessed by many important classes
of operators such as hyponormal operators and decomposable
operators \cite{7,14}. To emphasize the significance of Bishop's
property $(\beta)$; we mention the important connections to sheaf
theory and the spectral theory of several commuting operators from
the monograph by Eschmeier and Putinar \cite{9}. There are also
interesting applications to invariant subspaces \cite{9}, harmonic
analysis \cite{10}, and the theory of automatic continuity
\cite{15}. Unfortunately, but perhaps not surprisingly, the direct
verification of property ($\beta$) in con- crete cases tends to be
a difficult task. It is therefore desirable to have sufficient
conditions for property ($\beta$) which are easier to handle. Are
sums and products of commuting decomposable operators on Banach
spaces decomposable? This is one of the most important open
problems in the local spectral theory of operators on Banach
spaces. Similarly, it is not known if local spectral properties
such as Dunfords property ($C$), Bishops property $(\beta)$, or
the decomposition property $(\delta)$ are preserved under sums and
products of commuting operators. But it is shown in \cite{3} that
the single-valued extension property is not pre- served under the
sums and products of commuting operators; see also \cite{1}. On
the positive side, Sun \cite{23} proved that the sum and the
product of two commuting operators with Dunfords property (C) have
the single-valued extension property. Very recently, Aiena and
Muller \cite{2} showed that the (localized) single-valued
extension property is stable under commuting Riesz perturbations.
In this paper, we show that Bishops property ($\beta$), the
decomposition property ($\delta$), or decomposable operators are
stable under quasinilpotent, compact, and algebraic commuting
perturbations.

In \cite{12}, J.W. Helton initiated the study of operators $T\in
B(H)$ which satisfy an identity of the form
$$T^{*m}-\binom{m}{1}T^{*m-1}T+...+(-1)^{m}T^{m}=0.\eqno (1)$$
Further study of this class of operators is needed. Let $R,S\in
B(X)$ and let $C(R,S)^{k}(I): X\rightarrow X$ be defined by
$C(R,S)(A)= RA- AS$. Then
$$C(R, S)^{k}(I)= \sum^{k}_{j=0}(-1)^{k-j}\binom{k}{j}R^{j}S^{k-j}. \eqno (2)$$
Thus we have the following definition
\begin{definition}Let $R\in B(X)$. If there is an integer $k\geq
1$ such that an operator $S$ satisfies $C(R,S)^{k}(I)=0$, we say
that $S$ belongs to the Helton class of $R$. We denote this by
$S\in \mbox{Helton}_{k}(R).$
\end{definition}
We remark that $C(R, S)^{k}(I)=0$ does not imply $C(S,R)^{k}(I)=0$
in general.
\begin{example} \begin{displaymath}
S= \left( \begin{array}{ccccccc}
 0& A  &  B \\
 0& 0&0\\
 0&0&0
\end{array} \right),\,R= \left( \begin{array}{ccccccc}
 0& 0  &  C \\
 0& 0&D\\
 0&0&0
\end{array} \right),
\end{displaymath}
where $A,B,C$ and $D$ are bounded linear operators defined on $H$.
Then it is easy to calculate that $C(R,S)^{2}(I)=0$, but
$C(S,R){2}(I)\neq 0$.
\end{example}
\section{Main Results}
We will consider the Helton class of an operator which has
Bishop's property ($\beta$).
\begin{theorem} Let $R\in B(X$) has Bishop's property $(\beta)$ at $\lambda_{0}$. If
$S\in Helton_{k}(R)$, then $S$ has Bishop's property $(\beta)$ at
$\lambda_{0}$.
\end{theorem}
\begin{proof} Suppose that $R$ has Bishop's property $(\beta)$). Assume that there
exists $v > 0$ such that for every open subset $U \subset
D(\lambda_{0}; v)$ and for any sequence $\{f_{n}\}^{\infty}_{ n=1}
\subset O(U;X)$, $\lim_{n\rightarrow \infty}(\lambda I-
S)f_{n}(\lambda) = 0$ in $O(U;X)$. Since the terms of the below
equation equal to zero when $j + s \not= r$, it suffices to
consider only the case of $j + s = r$ Then we have Let
$f_{n}:U\rightarrow X$ be any sequence of analytic function on $U$
(any open set) such that $(\lambda I- S)f_{n}\lambda\rightarrow 0$
as $n\rightarrow 0$ uniformly on all compact subsets of $U$. Since
the terms of the bellow equation equal to zero when $j+s\neq r$,
it suffices to consider only the case of $j+s=r$. Then we have
$$\sum^{k}_{j=0}(-1)^{k-j}\binom{k}{j}(R-\lambda)^{j}(\lambda-S)^{k-j}$$
$$=\sum^{k}_{j=0}\sum^{j}_{r=0}\sum^{k-j}_{s=0}(-1)^{k-(s+r)}\binom{k}{j}\binom{j}{r}\binom{k-j}{s}R^{j}S^{k-j}$$
$$=\sum^{k}_{j=0}(-1)^{k-j}\binom{k}{j}R^{j}S^{k-j}.$$
Thus, we have
$$\sum^{k-j}_{j=0}\binom{k}{j}R^{j}S^{k-j}-(R-\lambda)^{k}f_{n}(\lambda)$$
$$=\sum^{k}_{j=0}\binom{k}{j}(R-\lambda)^{j}(\lambda -S)^{k-j}-(R-\lambda)^{k}f_{n}(\lambda)$$
$$=\sum^{k-1}_{j=0}\binom{k}{j}(R-\lambda)^{j}(\lambda -S)^{k-j}f_{n}(\lambda)$$
$$=\sum^{k-1}_{j=0}\binom{k}{j}(R-\lambda)^{j}(\lambda -S)^{k-j-1}(\lambda-S)f_{n}(\lambda)\rightarrow 0$$
as $n\rightarrow\infty$ uniformly on all compact subsets of $U$.
Since
$$\sum^{k}_{j=0}(-1)^{k-j}\binom{k}{j}R^{j}S^{k-j}=0,$$
we have $(R-\lambda)^{k}f_{n}(\lambda)\rightarrow 0$ as
$n\rightarrow \infty$ uniformly on all compact subsets of $U$.
Since $R$ has the Bishop's property ($\beta$),
$(R-\lambda)^{k-1}f_{n}(\lambda)\rightarrow 0$ as $n\rightarrow
\infty$ uniformly on all compact subsets of $U$. By induction we
get that $f_{n}(\lambda)\rightarrow 0$ as $n\rightarrow \infty$
uniformly on all compact subsets of $U$. Hence, $S$ has the
Bishop's property $\beta$.
\end{proof}
If $S; T\in B(X)$ have Bishop's property $(\beta)$ at
$\lambda_{0}$, does $S + T$ have Bishop's property $(\beta)$ at
$\lambda_{0}$? So far we do not know the answer about this
question. In the following theorem we study a special case of this
question.
\begin{theorem} If $R$ has Bishop's property $(\beta)$ at $\lambda_{0}$, $S \in
Helton_{k}(R)$ and $RS = SR,$ then T = R + S has Bishop's property
$(\beta)$ at $\lambda_{0}$.
\end{theorem}
\begin{proof}
It is easily seen that $C(2R; S)^{k}(I) = C(R; S)^{k}(I) = 0$.
Hence $T = R + S \in  Helton_{k}(2R)$. Since $2R$ has Bishop's
property ($\beta$), it follows from Theorem 2.1 that $T$ has
Bishop's property $(\beta)$.
\end{proof}

Note that if there exists an integer $k\in \mathbb{N}$ for which
$C(R; S)^{k}(I) = C(S;R)^{k}(I) = 0$, then the operators $S$ and
$R$ are said to be nilpotent equivalent. For $R$; $S\in B(X)$ with
$RS = SR$, it is easily seen that $C(R; S^{k}(I) = (R- S)^{}$ for
all $k \in \mathbb{N}$. Thus $R$ and $S$ are nilpotent equivalent
precisely when $R - S$ is nilpotent.

\begin{theorem} If $S; T \in B(X)$ are nilpotent equivalent, then $T$ has Bishop's
property ($\beta$) at $\lambda_{0}$ if and only if $S$ has
Bishop's property ($\beta$) at $\lambda_{0}$.
\end{theorem}
\begin{proof} It suffices to show by symmetry that
Bishop's property ($\beta$) is transferred from $S$ to $T$. Assume
that $S$ has Bishop's property ($\beta$) at $\lambda_{0}$, that
is, there exists $v > 0$ such that for every open subset $U
\subset D(\lambda_{0}; v)$ and for any sequence
$\{f_{n}\}^{\infty}_{ n=1} \subset O(U;X)$, $\lim_{n\rightarrow
\infty}(S - \lambda I)f_{n}(\lambda) = 0$ in $O(U;X)$ implies
$\lim_{n\rightarrow \infty} f_{n}(\lambda) = 0$ in $O(U;X)$. Then,
we have $$C(S, T)^{k}(I)f_{n}(\lambda)-(s-\lambda
I)^{k}f_{n}(\lambda)= C(S-\lambda I, T-\lambda
I)^{i}f_{n}(\lambda)- (S-\lambda I)^{k}f_{n}(\lambda)$$
$$=\sum^{k}_{i=1}(-1)^{i}(S-\lambda I)^{k-i}(T-\lambda
I)^{i}$$
$$= [\sum^{k}_{i=1}(-1)^{i}(S-\lambda I)^{k-i}(T-\lambda
I)^{i-1}](T-\lambda I)f_{n}(\lambda)\rightarrow 0$$
 in $O(U, X)$ as $n\rightarrow\infty$. Since $C(S, T)^{k}(I)=0$,
 we have $$\lim_{n\rightarrow\infty}(S-\lambda
 I)^{k}f_{n}(\lambda)=0$$ in $O(U, X)$.
 Since $S$ has the Bishop's property $(\beta)$ at $\lambda_{0}$,
 we have $$\lim_{n\rightarrow\infty}(S-\lambda
 I)^{k-1}f_{n}(\lambda)=0$$ in $O(U, X)$. By induction, we get
 $\lim_{n\rightarrow\infty}f_{n}{\lambda}=0$in $O(U, X)$. Hence
 $T$ has Bishop's property $(\beta)$ at $\lambda_{0}$.
\end{proof}
\begin{corollary} If $T$ has Bishop's property ($\beta$) at $\lambda_{0}$ and $N$ is nilpotent
such that $TN = NT$, then $T + N$ has Bishop's property $(\beta)$
at $\lambda_{0}$.
\end{corollary}
Unfortunately, but perhaps not surprisingly, the direct
verification of property ($\beta$) in concrete cases tends to be a
difficult task. It is therefore desirable to have sufficient
conditions for  Bishop's property ($\beta$) which are easier to
handle. in the following theorem we give a necessary and
sufficient condition for $2 \times 2$ operator matrix to have
Bishop's property ($\beta$).
\begin{theorem} Let $T_{1}\in B(X_{1})$ and let $T_{2}\in B(X_{2})$, where $X_{1}$ and
$X_{2}$ are two Banach spaces. If $T_{1}$ and $T_{2}$ have
Bishop's property ($\beta$) at $\lambda_{0}$, then $T_{1}\oplus
T_{2}$ has Bishop's property $(\beta)$ at $\lambda_{0}$.
\end{theorem}
\begin{proof} Assume that $T_{1}$ and $T_{2}$ have Bishop's property ($\beta$) at $\lambda_{0}$. Then
there exists $v_{i} > 0; i = 1; 2$, such that for every open
subset $U_{i}\subset D(\lambda_{0}; v_{i})$ and for any sequence
$\{f^{i}_{n}\}^{\infty}_{n=1} \subset O(U_{i};X_{i})$,
$\lim_{n\rightarrow \infty}(T_{i}-\lambda I)f^{i}_{n} (\lambda) =
0$ in $O(U_{i};X_{i})$ implies $\lim_{n\rightarrow \infty}
f^{i}_{n}(\lambda) = 0$ in $O(U_{i};X_{i})$. Let $f_{n} = f^{1}_{
n} \oplus f^{2}_{ n}\subset O(U;X_{1} \oplus X_{2})$ be an
analytic sequence, where $U \subset D(\lambda_{0}; min\{v_{1};
v_{2}\})$; $i = 1; 2$ is an arbitrary open subset and $f^{i}_{n}
\subset O(U;X_{i})$; $i = 1; 2$. Since $\lim_{n\rightarrow
\infty}(\lambda I- T_{1} \oplus T_{2})f_{n}(\lambda) = 0$ in
$O(U;X_{1} \oplus X_{2})$, then $lim_{n\rightarrow\infty}(\lambda
I - T_{i})f^{i}_{} (\lambda) = 0$ in $O(U;X_{i})$; $i = 1; 2$.
Since $T_{i}; i = 1; 2$ have Bishop's property $(\beta$), then
$\lim_{n\rightarrow\infty} f^{i}_{n} (\lambda) = 0$ in
$O(U;X_{i}); i = 1; 2$. Hence $lim_{n\rightarrow \infty}
f_{n}(\lambda) = 0$ in $O(U;X_{1} \oplus X_{2})$.
\end{proof}
Recall that an operator $K \in B(X)$ is said to be algebraic if
there exists a non trivial polynomial $p$ such that $p(K) = 0$.
Trivially, every nilpotent operator is algebraic and it is known
that every finite- dimensional operator is algebraic. It is also
known that every algebraic operator has a finite spectrum. In the
following theorem we study the Bishop's property ($\beta$) under
commuting algebraic perturbations.
\begin{theorem} Let $T \in B(X)$ and let $K \in B(X)$ be an algebraic
operator such that $TK = KT$. If $T$ has Bishop's property
($\beta$) at each of the zero of $p$, where $p$ is a non zero
polynomial for which $p(K) = 0$, then $T + K$ has Bishop's
property ($\beta$) at 0. In particular, if $T$ has Bishop's
property ($\beta$), then $T + K$ has Bishop's property $(\beta)$.
\end{theorem}
\begin{proof}Since $K$ is algebraic, $\sigma(K)$ is finite. Let
$\sigma(K)=\{\lambda_{1},...,\lambda_{n}\}$ and let $P_{i}$ the
spectral projection associated with  $K$ and the spectral set
$\{\lambda_{i}\}$. Set $Y_{i}=R(P_{i})$. Thus, $Y_{1},...,Y_{n}$
are closed linear subspaces of $X$ each of which is invariant
under both $K$ and $T$, and $X=Y_{1}\oplus,...,\oplus Y_{n}$ by
the classical spectral decomposition. Also, for arbitrary
$i=1,...,n$, $K_{i}=K|_{Y_{i}}$ and $T_{i}=T|_{Y_{i}}$ commute and
we have $\sigma(K_{i})=\{\lambda_{i}\}$. We claim that
$N_{i}=\lambda_{i}I-K_{i}$ is nilpotent for every $i=1,...,n$.
Since $p(\{\lambda_{i}\})=p(\sigma(K_{i}))=\{0\}$, it follows that
$p(\lambda_{i})=0$. Let $p(\mu)=(\lambda_{i}-\mu)^{r}q(\mu)$ with
$q(\lambda_{i}\neq 0)$. Then
$(\lambda_{i}-K_{i})^{r}q(K_{i})=p(K_{i})=0$ with $q(K_{i})$
invertible. Hence $(\lambda_{i}-K_{i})^{r}=0$ and
$N_{i}=\lambda_{i}-K_{i}$ is nilpotent for every $i=1,2...,n$. We
have $T_{i}-K_{i}=(T_{i}-\lambda I)-(K_{i}-\lambda_{i}
I)=T_{i}-\lambda_{i}I-N_{i}$. It is easily seen that if $T$ has
Bishop's property ($\beta$) at $\lambda_{0}$, then $T -
\lambda_{0} I$ has Bishop's property $(\beta)$ at 0. Since
Bishop's property ($\beta$) is inherited by restrictions to closed
invariant subspaces, we have $T_{i}-\lambda_{i}$ has Bishop's
property $(\beta)$ at 0. Therefore, $T_{i}-K_{i} =
T_{i}-\lambda_{i}- N_{i}$ has Bishop's property $(\beta)$ at 0 for
all $i = 1; ...; n$ by Corollary 2.1. Then it follows from Theorem
2.4 that $T-K = (T_{1}-K_{1})\oplus,...,\oplus(T_{n}-K_{n})$ has
Bishop's property $(\beta)$ at 0. The final claim can be
established by the main result
\end{proof}.
Now we will apply Theorem 2.5 to show that a decomposable operator
is stable under algebraic commuting perturbations.

\begin{corollary} Let $T \in B(X)$. If $T$ has the decomposition
property ($\beta$) and $K\in B(X)$ is an algebraic operator which
commutes with T, then $T + K$ has decomposition property
$(\delta)$.
\end{corollary}
\begin{proof} Assume that $T$ has the decomposition property
$(\delta)$, Then $T^{*}$ has Bishop's property ($\beta$). It is
obvious that $K^{*}$ is algebraic and commutes with $T^{*}$. Hence
$(T +K)^{*} = T^{*} +K^{*}$ has Bishop's property $(\beta)$.
Therefore, $T + K$ has decomposition property ($\delta$).
\end{proof}
\begin{corollary} Let $T\in B(X)$. If $T$ is decomposable and $K \in B(X)$ is an
algebraic operator which commutes with $T$, then $T +K$ is
decomposable.
\end{corollary}

 \begin{corollary} If $T$ is decomposable and $N$ is nilpotent
such that $TN = NT$, then $T + N$ is decomposable.
\end{corollary}
Recall \cite{5} that if $T\in B(X)$ is decomposable, and $f$ is an
analytic scalar-valued function on some neighborhood of
$\sigma(T)$, then $f(T)$ is decomposable. Thus, we have the
following corollaries.
\begin{corollary} If $T$ is decomposable and
$K$ is algebraic with $TK = KT$, then $f(T + K)$ is decomposable,
where $f$ is an analytic scalar-valued function on some
neighborhood of $\sigma(T)$.
\end{corollary}
\begin{corollary} If $T$ has Bishop's property ($\beta$) and $K$ is algebraic with $TK =
KT$, then $f(T + K)$ has Bishop's property ($\beta$), where $f$ is
an analytic scalar-valued function on some neighborhood of
$\sigma(T)$.
\end{corollary}
Under stronger hypotheses we can conclude that the decomposition
property ($\delta$) is preserved under certain types of
perturbations. In \cite{16}, it is shown that if $T$ is
decomposable and $S$ is an operator that commutes with T such that
S has totally disconnected spectrum, then $T + S$ is decomposable.
By using this, we have the following immediate corollary.

\begin{corollary} Let $T$ be a decomposable operator on $B(X)$, and
let $S \in B(X)$ that commutes with $T$. If

 1. $S$ is a quasinilpotent operator, or

 2. $S$ is a compact
operator, or

3. $S$ has discrete spectrum, then $T + S$ is decomposable. In
particular $T + S$ has decomposition property $(\delta)$.
\end{corollary}
\textbf{acknowledgements}\\
 \textbf{Data Availability Statement} No data were used to support this study.
{}
\end{document}